\documentclass[graybox]{svmult}

\usepackage{mathptmx}       
\usepackage{helvet}         
\usepackage{courier}        
\usepackage{type1cm}        
\usepackage{makeidx}         
\usepackage{graphicx}        
\usepackage{multicol}        
\usepackage[bottom]{footmisc}

\makeindex             

\usepackage{amsmath,amsfonts,amssymb,latexsym,epsfig}
\include{references}
\usepackage{mathrsfs}
\usepackage{verbatim}
\usepackage{latexsym}
\usepackage{amssymb}
\usepackage{amsbsy}
\usepackage{theorem}
\usepackage{enumerate}
\usepackage{times}

\theoremstyle{prop}
\newtheorem{prop}[theorem]{Proposition}

\theoremstyle{assumption}
\newtheorem{assumption}[theorem]{Assumption}

\DeclareMathOperator*{\essup}{ess-sup}

\newcommand{\Bf}{\overline \Phi}
\newcommand{\Bi}{\overline I}
\newcommand{\bu}{\overline u}
\newcommand{\ps}{p^{\star}}

\newcommand{\de}{\delta^{\epsilon}}
\newcommand{\dd}{\delta_1^{\epsilon}}

\newcommand{\eS}{S^{d}}
\newcommand{\pS}{S^{d,+}}

\newcommand{\Ea}{E_{\rm {ad}}}

\newcommand{\bbR}{\mathbb R}
\newcommand{\bbP}{\mathbb P}
\newcommand{\bbE}{\mathbb E}
\newcommand{\bP}{\mathbb P}
\newcommand{\bbZ}{\mathbb Z}

\newcommand{\bbT}{{\mathbb T}}
\newcommand{\ve}{v^{\epsilon}}
\newcommand{\vo}{v_{0}}
\newcommand{\po}{p_{0}}
\newcommand{\pe}{p^{\epsilon}}
\newcommand{\pea}{\pe_{\rm a}}
\newcommand{\Ke}{K^{\epsilon}}

\newcommand{\xe}{x^{\epsilon}}
\newcommand{\xo}{x_{0}}
\newcommand{\Kb}{{K_0}}
\newcommand{\ue}{u^{\epsilon}}
\newcommand{\uo}{u_0}

\newcommand{\eps}{\epsilon}

\newcommand{\Ex}{{\mathbb E}}

\newcommand{\T}{{\mathbb T}}
\newcommand{\R}{{\mathbb R  }}

\newcommand{\cL}{\mathcal L}

\newcommand{\be}{\begin{equation}}
\newcommand{\ene}{\end{equation}}
\newcommand{\cG}{{\cal G}}
\newcommand{\cC}{{\cal C}}
\newcommand{\cF}{{\cal F}}
\newcommand{\cFn}{{\cal F}_{\eta}}
\newcommand{\xt}{x_{\rm init}}

\newcommand{\commentout}[1]{}

\definecolor{darkred}{rgb}{.7,0,0}

\definecolor{green}{rgb}{0,0.7,0}

\begin{document}

\title*{Multiscale Modelling and Inverse Problems}
\author{J. Nolen, G.A. Pavliotis and A.M. Stuart}
\institute{
           J. Nolen \at Department of Mathematics, Duke University 
           Durham, NC 27708, USA , \email{nolen@math.duke.edu}
\and 
     G.A. Pavliotis \at Department of Mathematics  Imperial College London 
        London SW7 2AZ, UK \email{g.pavliotis@imperial.ac.uk}
\and        
     A.M. Stuart \at   Mathematics Institute 
        Warwick University 
        Coventry CV4 7AL, UK \email{A.M.Stuart@warwick.ac.uk}   
        }
%
%
\maketitle

\abstract{
The need to blend observational data and mathematical models arises in many applications
and leads naturally to inverse problems. Parameters appearing
in the model, such as constitutive tensors, initial conditions, boundary conditions, and forcing can be estimated on the basis of observed data. The resulting inverse problems are often ill-posed and some form of regularization is required. 
These notes discuss parameter estimation in situations where the unknown parameters vary across multiple scales.  We  illustrate the main ideas using a simple model for groundwater flow.\newline\indent
We will highlight various approaches to regularization
for inverse problems,
including Tikhonov and Bayesian methods.
We illustrate three ideas that arise when considering
inverse problems in the multiscale context. The
first idea is that the choice of space or set in which to 
seek the solution to the inverse problem is intimately related to
whether a homogenized or full multiscale 
solution is required. This is a choice of regularization.
The second idea is that, if a homogenized solution
to the inverse problem is what is desired, then this 
can be recovered from carefully designed
observations of the full multiscale system. 
The third idea is that the theory of
homogenization can be used to improve the estimation
of homogenized coefficients from multiscale data.}


\section{Introduction}
\label{sec:intro}

The objective of this overview is to demonstrate the important role
of multiscale modelling in the solution of inverse problems
for differential equations. 
The main inverse problem we discuss is that of determining unknown parameters by matching observed
data to a differential equation model involving those parameters. The unknown parameters may be functions, in general, and they may have variation over multiple (length) scales. This multiscale structure makes the forward problem more challenging: numerically computing the solution to the differential equation requires very high resolution. The multiscale structure also complicates the inverse problem. Should we try to fit the data with a high-dimensional parameter, or should we seek a low-dimensional ``homogenized" approximation of the parameter? If a low-dimensional parameter model is used, how should we account for the mismatch between the true parameters and the low-dimensional representation? After obtaining a solution to the inverse problem, one typically wants to make further predictions using whatever parameter is fit to the observed data, so it is important to consider whether a low-dimensional representation of the unknown parameter is sufficient to make additional predictions.

Throughout these notes the unknown parameters will be denoted by $u \in X$; typically $u$ is a function assumed to lie in a Banach space $X$. We use $y \in Y$ to denote the data (for simplicity we often take $Y=\bbR^N$) and $z$ to denote the predicted quantity,
assumed to be an element of a Banach space $Z$ or,
in some cases, a $Z-$valued random variable. 
The map $\cG:X \to \bbR^N$ denotes the forward 
mapping from the unknown parameter to the data, 
and $\cF:X \to Z$ (or $\cF: X \times \Omega \to Z$
in the random case) denotes the forward
mapping from the parameter to the prediction. 
We sometimes refer to $\cG$ as the 
{\em observation operator} and $\cF$ as
the {\em prediction operator}. Both $\cG$ and
$\cF$ are typically derived from a common
solution operator $G:X \to P$ mapping $u \in X$
to the solution $G(u) \in P$ of a partial differential
equation (PDE), where $P$ is a Banach space. For example $\cG$ may be derived
by composing $G$ with $N$ linear functionals.

The ideal inverse problem is to determine 
$u \in X$ from knowledge of
$y \in \bbR^N$ where it is assumed that $y=\cG(u).$
In practice, however, the data $y$ is generated from outside
this clean mathematical model, so it is natural
to think of the data $y$ as being given by 
\begin{equation}
\label{eq:data}
y=\cG(u)+\xi
\end{equation}
for some $\xi \in \bbR^N$ quantifying model 
error\footnote{Model error can be incorporated within
the set of unknown parameters $u$ and estimated
using data; however this idea is not pursued
here.} and
observational noise. The value of $\xi$ is not known,
but it is common in applications to assume
that some of its statistical properties are known
and these can then be built into the methods used to
estimate $u$. Once the function $u$ is determined
by solving this inverse problem, it can be
used to make a prediction $z=\cF(u).$

We illustrate three ideas that arise when
attempting to solve the inverse problem
defined by \eqref{eq:data} 
in the multiscale context:

\begin{itemize}

\item (a) 
The choice of the space or set in which to 
seek the solution to the inverse problem is intimately related to
whether a low-dimensional ``homogenized" solution or a high-dimensional ``multiscale" 
solution is required for predictive capability. This is a choice of regularization.

\item  
(b) If a homogenized solution
to the inverse problem is desired, then this 
can be recovered from carefully designed
observations of the full multiscale system.

\item  
(c) The theory of
homogenization can be used to improve the estimation
of homogenized parameters from observations of multiscale data.

\end{itemize}

In Section \ref{sec:for} we consider in detail a worked example which exemplifies the use of
multiscale methods to approximate the forward
problems $\cG$ and $\cF$ for data and predictions;
this example will be used to illustrate many
of the general ideas developed in these notes,
and the three ideas (a)--(c) in particular.
Section \ref{sec:inv}
is devoted to a brief overview of 
regularization techniques for inverse problems,
and to discussion of the idea (a).
Section \ref{sec:invm} is devoted to the 
idea (b). We study the problem of estimating
a single scalar parameter in a homogenized
model of groundwater flow, given data
which is generated by a full multiscale
model. This may be seen as a surrogate
for understanding the use of real-world
data (which is typically multiscale in
character) to estimate parameters in simpler
homogenized models. Section
\ref{sec:invm2} is devoted to the idea (c). We
study the use of
ideas from multiscale methodology to enhance
parameter estimation techniques for
homogenized models. The viewpoint
taken is that the statistics of the
error $\xi$ appearing in \eqref{eq:data}
can be understood using the theory of homogenization for random media; when these statistical
properties depend on the unknown parameter
$u$ the noise $\xi$ is no longer additive and
its dependence on $u$ plays an important role in the
parameter estimation process.

\subsection{Notation}
\label{ssec:not}

The following notation will be used throughout.
We use $|\cdot|$ to denote
the Euclidean norm on $\bbR^m$ (for possibly
different choices of $m$).
We let $\eS$ (resp. $\pS$) denote the set of symmetric 
(resp. positive-definite) second order
tensors on $\bbR^d$. If $\Gamma \in \pS$, we define the weighted norm $|\cdot|_{\Gamma}=|\Gamma^{-\frac12}\cdot|$ on $\bbR^m$. Throughout the notes, $X$ is a Banach space, containing the
functions that we wish to estimate, and $E$ a Banach
space compactly embedded into $X$. When studying the inverse problem from a Bayesian perspective we will use Gaussian priors
on $X$, defined via a covariance operator $\cC$
on a Hilbert space $H \supseteq X$, 
with norm $\|\cdot\|_{H}$. 
In this situation $E$ will be
the Hilbert space with norm $\|\cC^{-\frac12}\cdot\|_{H}$.
%
%
\subsection{Running Example}
\label{ssec:g}

We consider a model for groundwater flow in a medium with permeability tensor $k$,
pressure $p$ and Darcy velocity $v$ (or the volume flux of water per unit area) related to the pressure via the Darcy law: 
\begin{equation}\label{e:darcy}
v = -\frac{k}{\mu} ( \nabla p - \rho g \hat{e}_z)
\end{equation}
where $\mu$ is the fluid viscosity, $\rho$ is the fluid density, $g$ is the acceleration due to gravity and $\hat{e}_z$ is the unit vector in the $z$-direction. We choose units in which $\mu=1.$ We also assume that we have a constant density fluid and redefine the pressure by adding $\rho g z$ ($z$ is the vertical direction) to write~\eqref{e:darcy} in the form $v= -k \nabla p$. Assuming that the Darcy velocity is divergence-free, except at certain
known source/sink locations,
we obtain the following elliptic equation for the pressure: 
\begin{align}
\begin{split}
\label{eq:pde1}
\nabla \cdot v&=f,\quad x \in D,\\
p&=0, \quad x \in \partial D,\\
v&=-k \nabla p
\end{split}
\end{align}
where $D \subset \bbR^d$ is an open and bounded set with regular boundary, and $f$ is assumed to be known. The permeability tensor field $k$, 
however, is assumed to be
unknown and must be determined from data.
In order to make the elliptic PDE \eqref{eq:pde1}
for the pressure $p$ well-posed, we assume that
the permeability tensor $k(x)$ is an element
of $\pS$ and so we write it as the (tensor) 
exponential: $k(x)=\exp\bigl(u(x)\bigr)$, $u \in \eS$.
It is natural to view $u$ as an element of
$X:=L^{\infty}(D;\eS)$ and to consider
weak solutions of \eqref{eq:pde1} 
with $f \in H^{-1}(D).$ 
Then we have a unique solution $p \in H^1_0(D)$
satisfying 
\begin{equation}
\label{eq:pest}
\|\nabla p\|_{L^2} \le c_1\exp(\|u\|_{X})\|f\|_{H^{-1}},
\end{equation}
for some $c_1>0$ depending only on $d$ and $D$, and $\| u\|_X$ being the essential supremum of the spectral radius of the matrix $u(x)$, as $x$ varies over $D$:
$$
\| u\|_X = \essup_{x \in D} \left( \max_{\substack{\xi \in \bbR^d \\ \lvert \xi \rvert = 1}} \; \lvert u(x) \xi \rvert \right).
$$
Thus we may define $G:X \to H^1_0(D)$ by $G(u)=p.$
Now consider a set of real-valued continuous linear
functionals $\ell_j:H^1(D) \to \bbR$ and define
$\cG:X \to \bbR^N$ by $\cG(u)_j=\ell_j(G(u)).$
The inverse problem is to determine $u \in X$ from
$y \in \bbR^N$ where it is assumed that $y$ is given
by \eqref{eq:data}. 
Using \eqref{eq:pest} one may show that $G:X \to H^1_0(D)$ (resp.
$\cG:X \to \bbR^N$) is Lipschitz.
Indeed if $p_i$ denotes  the solution to
\eqref{eq:pde1} with log permeability $u_i$ 
then,
we have
\begin{equation}
\label{eq:pest2}
\|\nabla p_1-\nabla p_2\|_{L^2} \le (c_1)^2 \| u_1 - u_2 \|_X \exp\Bigl(
2 (\|u_1\|_{X} + \|u_2\|_{X}) \Bigr) \,   \|f\|_{H^{-1}}.
\end{equation}

Study of the transport of contaminants in groundwater
flow is a natural example of a useful prediction that
can be made once the inverse problem is solved. To model
this scenario we consider a particle $x(t) \in \bbR^d$
which is advected by the the groundwater velocity field $v/\phi$, where $\phi$ is the porosity of the rock and $v$ is the Darcy velocity field
from \eqref{eq:pde1}, and subject to diffusion with coefficient $2\eta.$
Assuming that the contaminant is initially at
$\xt$ we obtain the stochastic differential equation
(SDE):
\begin{equation}
dx=\frac{v(x)}{\phi} \, dt+\sqrt{2\eta} \, dW, \quad x(0)=\xt,
\end{equation}
where $W(t)$ is a standard Brownian motion on $\R^d$.
If we are interested in predicting the
location of the contaminant at time $T$
then our prediction will be the function $\cFn$
given by $\cFn(u)=x(T).$ Here for each fixed
$\eta \in [0,\infty)$ the function $\cFn$
maps $X$ into the family of $\bbR^d-$valued random
variables.

\section{The Forward Problem: Multiscale Properties}
\label{sec:for}

Some inverse problems arising in applications
have the property that
the forward model $\cG$ mapping the unknown to the
data will produce similar output on both highly
oscillatory functions $u$ and on 
appropriately chosen smoothly varying
functions $u$. Furthermore, for some choices
of prediction function $\cF$ the predictions
themselves will also be close for
both highly
oscillatory functions $u$ and on 
appropriately chosen smoothly varying
functions $u$.
These properties can be seen from an application
of multiscale analysis, and we illustrate them by considering the 
problem introduced in Section 
\ref{ssec:g}. There are many texts on the theory of multiscale analysis. For example, the basic homogenization theorems discussed here are developed in \cite{lions}.
A recent overview of the subject, with many other references and using the
same notational conventions that we adopt here, is 
\cite{PS08}. 

We consider a multiscale version of
the running example from Section \ref{ssec:g} 
where the permeability tensor is $k = \Ke(x)=K(x,x/\epsilon)$
where $K:D\times \bbT^d \to \pS$ is periodic in the second argument, $\epsilon > 0$ a small parameter. 
For now we have assumed periodic dependence on
the fast scale in $\Ke$; however we will
generalize this to random dependence in later
developments.

With this permeability we obtain the family of problems
\begin{subequations}
\label{eq:pde22}
\begin{eqnarray}
\nabla \cdot \ve &=& f,\quad x \in D,\\
\pe &=& 0, \quad x \in \partial D,\\
\ve &=&  -\Ke \nabla \pe.
\end{eqnarray}
\end{subequations}
If we set $\eta = \epsilon \eta_0$, then the transport 
of contaminants is given by the SDE
\begin{equation}
d\xe=\frac{\ve(\xe)}{\phi}\, dt+\sqrt{2\eta_0 \epsilon} \, dW, \quad \xe(0)=\xt.
\label{eq:sde22}
\end{equation}

Standard techniques from the theory of homogenization for
elliptic PDEs can be used to show that for $\epsilon$ small,
\begin{equation}
\label{eq:pap}
\pe(x) \approx \pea(x):=p_0(x)+\epsilon p_1(x,\frac{x}{\epsilon})
\end{equation}
where $p_0$ and $p_1$ are defined as follows.
First we define the effective (homogenized) permeability
tensor $\Kb$ via solution of the 
{\em cell problem} for $\chi(x,y)$:
\begin{equation}
\label{eq:cell}
-\nabla_{y} \cdot \bigl(\nabla_{y}\chi K^T\bigr)=\nabla_{y} \cdot K^T,\quad y \in 
\bbT^d.
\end{equation}
Then
\begin{align}
\label{eq:K}
\Kb(x)&=\int_{\bbT^d} Q(x,y)dy,\\
\label{eq:Q}
Q(x,y)&=K(x,y)+K(x,y)\nabla_{y}\chi(x,y)^T.
\end{align}
In this sense we observe that the effective
diffusivity $\Kb(x)$ is the average of $Q(x,y)$
over the fast scale $y$. This is {\em not} 
equal to the average of $K(x,y)$ over $y$, 
except in trivial cases.  
We denote by $\uo$ the logarithm of $K_0$ so that
$K_0=\exp(\uo).$

The function $p_0$ solves the ($\epsilon$ independent) elliptic PDE
\begin{subequations}
\label{eq:pde2}
\begin{eqnarray}
\nabla \cdot v_0 &=&  f,\quad x \in D,\\
p_0 &=& g, \quad x \in \partial D,\\
v_0 &=& -\Kb \nabla p_0.
\end{eqnarray}
\end{subequations}
and the corrector $p_1$ is given by
\begin{equation}
\label{eq:p1}
p_1(x,y)=\chi(x,y) \cdot \nabla p_0(x).
\end{equation}
Note that \eqref{eq:cell} may be written
as
\begin{equation}
\label{eq:Qprop}
-\nabla_{y} \cdot \bigl(Q^T\bigr)=0,\quad y \in 
\bbT^d.
\end{equation}
This shows that $Q$, which is averaged to give the effective permeability tensor, is divergence-free with respect to the fast variable $y.$

It is possible to prove that, in the limit as $\eps \rightarrow 0$, solutions to~\eqref{eq:pde22} converge to solutions to~\eqref{eq:pde2}, the convergence being strong in $L^2(D)$ and weak in $H^1(D)$~\cite{cioran, allaire, PS08}. However if we want to prove strong
convergence in $H^1$ then we need to include
information about the corrector term $p_1$. The
following theorem and corollary summarize these
ideas. For proofs see \cite{allaire}, or the discussion
in the texts \cite{cioran,PS08}.

\begin{theorem}  
Let $\pe$ and $\po$ be the solutions of~\eqref{eq:pde22} and~\eqref{eq:pde2}. Assume that $f \in C^{\infty}(D)$ and that $K(x,y) \in C^{\infty}(D ; C^{\infty}_{\rm{per}}(\T^d))$. Then
\begin{equation}
\lim_{\epsilon \rightarrow 0}
\|\pe-\pea\|_{H^{1}}=0. 
\label{e:corr}
\end{equation}
\label{t:homcon}
\end{theorem}

\begin{corollary}\label{cor:homcon}
Under the same conditions as in Theorem \ref{t:homcon} we have \label{c:homcon}
$$\|\pe-p_0\|_{L^2} \to 0\quad {\rm and}\quad \|\nabla \pe-\bigl(I+\chi_y(\cdot,\cdot/\epsilon)^T\bigr)\nabla p_0\|_{L^2} \to 0$$
as $\epsilon \to 0.$
\end{corollary}

In fact it is frequently the case that the convergence in Theorem \ref{t:homcon} may be obtained in a stronger topology. Reflecting this we make the following assumption.
\begin{assumption} \label{ass:1}
The function $\pe$ converges to $p_0$
in $L^{\infty}(D)$ and its gradient
converges to the gradient of $p_0+\epsilon p_1$
in $L^{\infty}(D)$ so that
\begin{equation*}
\lim_{\epsilon \rightarrow 0}
\|\pe-\pea\|_{W^{1,\infty}}=0.
\end{equation*}
\end{assumption}

In Appendix~\ref{app:b} we prove this assumption for the one dimensional version of \eqref{eq:pde22}. The proof in the multidimensional case will be presented elsewhere~\cite{NPS10}. The proof of this assumption in the multidimensional case is based on the estimates proved in~\cite{AL87} (in particular, Lemma 16), see also~\cite[Lemma 2.1]{EHW2000}.

With these limiting properties of the elliptic problem~\eqref{eq:pde22} at hand it is natural to ask what is the limiting behaviour of $\xe$ governed by \eqref{eq:sde22}. To answer this question we define 
\begin{equation}
\frac{d \xo}{d t}=\frac{\vo(\xo)}{\phi}, \quad \xo(0)=\xt.
\label{eq:ode22}
\end{equation}
Notice that this ordinary differential equation
(ODE) has vector field $\vo$ which is defined
entirely through knowledge of the homogenized
permeability $K_0$: once $K_0$ is known, the
elliptic PDE \eqref{eq:pde2} can be solved
for $p_0$ and then $v_0$ is recovered from
(\ref{eq:pde2}c). 
If we can show that solutions of \eqref{eq:sde22}
and \eqref{eq:ode22} are close then this will
establish that the prediction of particle
transport in the model \eqref{eq:pde22},
\eqref{eq:sde22} can be made accurately by use of
only homogenized information about the permeability.

In proving such a result there are a number of
technical issues which arise caused by the presence
of the boundary $D$ of the domain in which
the PDE \eqref{eq:pde22} is posed. In particular solutions
of \eqref{eq:sde22} may leave $D$ requiring a definition
of the velocity field outside $D$. These issues
disappear if we consider the case where $D$ is itself
a box of length $L$ and is equipped with periodic
boundary conditions instead of Dirichlet conditions: we
may then extend all fields to the whole of $\bbR^d$ by
periodicity. In this case, the homogenization theory for \eqref{eq:pde22}
with (\ref{eq:pde22}b) replaced by periodic boundary
conditions is identical to that given above, except that
(\ref{eq:pde2}b) is also replaced by periodic boundary
conditions.  We write $D=(L\bbT)^d$ and adopt this
periodic setting for the next theorem, which is proved in Appendix \ref{app:a}:

\begin{theorem} \label{t:avcon}
Let $\xe (t)$ and $\xo (t)$ be the solutions to 
equations~\eqref{eq:sde22} and~\eqref{eq:ode22}, with
velocity fields extended from $D=(L\bbT)^d$ to $\bbR^d$
by periodicity, and assume that 
Assumption~\ref{ass:1} holds. 
Assume also that $f \in C^{\infty}(D)$ and that $K(x,y) \in C^{\infty}(D ; C^{\infty}_{\rm{per}}(\T^d))$. Then
$$\lim_{\epsilon \to 0} \bbE \sup_{0 \le t \le T}
\|\xe(t)-\xo(t)\|=0.$$
\end{theorem}

In summary, this example exhibits the property that, if the length scale $\epsilon$ is small, the data generated from $K^\epsilon$ and $K_0$ may appear very similar due to homogenization effects. Therefore, when trying to infer parameters from data, it is difficult to distinguish between $K^\epsilon$ and $K_0$ without some form of regularization or prior assumptions about the form of the parameter. On the other hand, Theorem \ref{t:avcon} shows that knowing only $K_0$ is sufficient to make accurate predictions of the trajectories of (\ref{eq:sde22}).

\section{Regularization of Inverse Problems}
\label{sec:inv}

In this section we describe various approaches
to regularizing inverse problems,
motivating them by reference to the
multiscale example in the previous section.
The approach to regularizing
which is described in
Section \ref{ssec:reg1} is developed in
detail in \cite{BK89}. 
The Tikhonov regularization approach
from Section \ref{ssec:reg2} is developed in
detail in \cite{ehn96, Fitzp91}. Both of these regularization
approaches are specific examples of
the general set-up often called PDE
constrained optimization, which we discuss 
in Section \ref{ssec:pdec}; this
subject is overviewed in \cite{HPUU09}. 
An overview of the Bayesian approach to inverse
problems, a subject that we outline in Section
\ref{ssec:reg3}, is given in \cite{Stuart10} and \cite{Fitzp91}. 

\subsection{Set-Up}
\label{ssec:setup}

Our objective here is to determine $u$, given $y$, where $u$ and $y$ are 
related by \eqref{eq:data}. We assume that,
whilst the actual value of $\xi$ is not available,
it is reasonable to view it as a single draw from a statistical
distribution whose properties are known to us. To be
concrete we assume that $\xi$ is
drawn from a mean zero Gaussian random variable 
with covariance $\Gamma$: we write this as
$\xi \sim N(0,\Gamma)$. 
We make the following continuity assumption 
concerning the observation operator $\cG$.
Note that this (local) Lipschitz condition also implies
an (exponential in $\|u\|_{X}$) bound on $|\cG(u)|.$

\begin{assumption} 
\label{ass:g}
There are constants $c_1, c_2>0$
such that, for $u_i \in X$ with 
$\|u_i\|_{X}<r, i=1,2$, 
$$|\cG(u_1)-\cG(u_2)| \le c_1 \exp(c_2 r)\|u_1-u_2\|_{X}.$$
\end{assumption}

In general the inverse problems such as that given by
\eqref{eq:data} with $\xi=0$ are hard to solve: 
they may have no solutions, multiple solutions and
solutions may exhibit sensitive dependence on initial
data. For this reason it is natural to seek
a least squares approach to finding functions $u$
which best explain the data. In view of the
assumed structure on $\xi$ a natural least squares
functional is
\begin{equation}
\label{eq:phi}
\Phi(u)=\frac12|y-\cG(u)|_{\Gamma}^2.
\end{equation}
The weighting by $\Gamma$ in the Euclidean
norm induces a normalization on the model-data
mismatch. This normalization is given by
the assumed standard deviations of the noise in
a coordinate system defined by the eigenbasis for
$\Gamma$.

\begin{example} \label{ex:rev}
Consider the running example of
Section \ref{ssec:g}.
Equation \eqref{eq:pest2} shows that
Assumption \ref{ass:g} holds in this case, noting that
$\cG(u)_j=\ell_j(p)$ for some linear functional $\ell_j$
on $H^1(D)$, 
with the choice $X=L^{\infty}(D;\eS)$, provided
$f \in H^{-1}$.
We use this example to illustrate why
inverse problems are, in general, hard. 

Assume
that the linear functionals $\ell_j$ satisfy the
property that $\ell_j(\pe-\po) \to 0$ as $\epsilon
\to 0.$ This occurs if they are 
linear functionals on $L^2(D)$, by Theorem
\ref{t:homcon} or
if Assumption \ref{ass:1} holds, if they are linear
functionals on $C({\overline D}).$
Writing this in terms of $\cG$ we have
$|\cG(\ue)-\cG(\uo)| \to 0$ as $\epsilon \to 0.$ 
(Note that this occurs even though $\ue$ and $\uo$ are not
themselves close.)
Hence there is
an uncountable family of functions (indexed
by all $\eps$ sufficiently small) which all
return approximately the same value of $\Phi(\ue)$
and thus simply minimizing $\Phi$ may be very difficult.
Furthermore, there may be minimizing
sequences which do not converge. For example
fix a particular realization of
the data given by $y=\cG(\uo)$ where $\uo$
is the homogenized log permeability. 
Then $\Phi(\ue) \geq 0$ for all $\epsilon > 0$ and $\Phi(\ue) \to 0$ as $\eps \to 0$, since 
\begin{align}
|\Phi(\ue)| & = \frac12|y-\cG(\ue)|_{\Gamma}^2 = \frac12|\cG(\uo) -\cG(\ue)|_{\Gamma}^2
\end{align}
On the other hand, $u^\epsilon$ does not converge in $X$ as $\epsilon \to 0$.

\commentout{
Since $\ue$ is bounded in $X$ by Assumption \ref{ass:g} we deduce
that $\Phi(\ue) \to 0$ because $\cG(\ue) \to \cG(\uo)$.
On the other hand, if we choose $E=H^1_0(D)$,
then $\|\ue\|_{E}={\cal O}(\epsilon^{-1})$ 
as $\epsilon \to 0.$
Thus there exist minimizing sequences which do
not converge in $E$. This choice of $E$ is natural as
it imposes weak differentiability
of the velocity field $\ve$ appearing in \eqref{eq:sde22}.
}
\noindent\qed\end{example}

In order to overcome the difficulties demonstrated
in this example regularization is needed.
In the remaining sections we discuss various
regularizations, in general, illustrating ideas
by returning to the running example.

\subsection{Regularization by Minimization Over a Convex, Compact Set}
\label{ssec:reg1}

Recall that $E$ is a Banach space compactly embedded
into $X$.
Let $\Ea=\{u \in E:\|u\|_{E} \le \alpha\}.$
Then $\Ea$ is a closed convex and bounded set in $E$
and, as such, any sequence in $\Ea$ must contain
a weakly convergent subsequence with limit in $\Ea$
(see, for example, Theorem 1.17 in \cite{HPUU09}).
Now consider the minimization problem
\begin{equation}
\label{eq:inf1}
\Bf=\inf_{u \in \Ea} \Phi(u).
\end{equation}

\begin{theorem}
\label{t:inf1}
Any minimizing sequence $\{u^n\}_{n \in \bbZ^+}$
for \eqref{eq:inf1} contains a weakly convergent
subsequence in $E$ with limit $\bu \in \Ea$
which attains the infimum: $\Phi(\bu)=\Bf.$
\end{theorem}

\begin{proof} 
This is a classical theorem
from the field of optimization; see
\cite{HPUU09} for details and context.
Since $\{u^n\}$ is contained in
$\Ea$ we deduce 
the existence of a subsequence (which for
convenience we relabel as $\{u^n\}$)
with weak limit $\bu \in \Ea.$
Thus $u^n \rightharpoonup \bu$ in $E$. Hence,
by compactness, $u^n \to u$ in $X$. By
Assumption \ref{ass:g} we deduce that $\Phi:E \to \bbR$
is weakly continuous. By definition, for any
$\delta>0$ there exists $N=N(\delta)$ such that
$$\Bf \le \Phi(u_n) \le \Bf+\delta,\quad 
\forall n \ge N.$$
By weak continuity of $\Phi:E\to\bbR$ we deduce that
$$\Bf \le \Phi(\bu) \le \Bf+\delta.$$
The result follows since $\delta$ is arbitrary.
\qed
\end{proof}

\begin{example} \label{ex:g99}
Consider the running example of Section
\ref{ssec:g}. Let $A$ denote a
fixed symmetric positive-definite tensor $A$ so that
$\log(A)$ is defined. We define the subspace
of tensor valued functions of the form $u'=uI+\log(A)$,
for some constant $u\in\bbR$ noting that then
$\exp(u')=\exp(u)A$. By Lipschitz continuity of ${\cal G}$ in
$u' \in X$ we deduce (abusing notation)
Lipschitz continuity of ${\cal G}$ viewed as a function 
of $u \in \bbR$. We define
\begin{equation}
\label{eq:space}
\Ea=\{u\in\bbR: |u| \le \alpha\}.
\end{equation} 
We may take the norm $\|\cdot\|_{E}=|u|.$
Thus the problem \eqref{eq:inf1} 
attains its infimum for some $\bu \in \Ea$.
The regularization of seeking to minimize $\Phi$ over
$\Ea$ corresponds to looking for solution over
a one-parameter set of tensor fields, in which the
free parameter is bounded by $\alpha.$ Note that such
a solution set automatically rules out the
oscillating minimizing sequences which were exhibited
in Example \ref{ex:rev}.
\noindent\qed\end{example}

\subsection{Tikhonov Regularization}
\label{ssec:reg2}

Instead of regularizing by seeking to minimize
$\Phi$ over a bounded and convex subset of
a compact set $E$ in $X$, we may instead
adopt the Tikhonov approach to regularization.
We consider the minimization problem
\begin{equation}
\label{eq:inf2}
\Bi=\inf_{u \in E} I(u),
\end{equation}
where
\begin{equation}
\label{eq:IT}
I(u)=\frac{\lambda}{2}\|u\|_E^2+\Phi(u).
\end{equation}

\begin{theorem}
\label{t:inf2}
Any minimizing sequence $\{u^n\}_{n \in \bbZ^+}$
for \eqref{eq:inf2} contains a weakly convergent
subsequence in $E$ with limit $\bu$
which attains the infimum: $I(\bu)=\Bi.$
\end{theorem}

\begin{proof} This is a classical theorem
from the calculus of variations; see
\cite{Dac89} for details and context.
Since $\{u^n\}$ is a minimizing sequence and $\Phi \geq 0$, 
we deduce that for any
$\delta>0$ there exists $N=N(\delta)$ such that
$$\frac{\lambda}{2}\|u_n\|_{E}^2 \le \bar I +\delta,\quad
\forall n \ge N.$$
From this it follows that $\{u^n\}_{n \in \bbZ^+}$ is
bounded in $E$ and hence contains a weak limit $\bu$,
along a subsequence which, for
convenience, we relabel as $\{u^n\}$.
The weak continuity of $\Phi:E \to \bbR$,
together with weak lower semicontinuity of the
function $\|\cdot\|_{E}^2 \to \bbR$ implies the
weak lower semicontinuity of $I:E \to \bbR$.
Hence
$$I(\bu) \le \liminf_{n \to \infty} I(u_n) \le \Bi.$$
Since $I(\bu) \ge \Bi$, the result follows.
\qed
\end{proof}

\begin{example} \label{ex:g2}
Consider the running example of Section
\ref{ssec:g}. Let $E=H^s(D;\eS)$ and note that $E$
is compact in $X=L^{\infty}(D;\eS)$ for $s>d/2.$
Thus the problem \eqref{eq:inf2} 
attains its infimum for some $\bu \in E$.
As with the example from the previous section
the regularization rules out highly oscillating
minimizing sequences such as those seen in Example
\ref{ex:rev}. The choice of the parameter $\lambda$
will effect how much oscillation is allowed in any
minimizing sequence.
\noindent\qed\end{example}

\subsection{PDE Constrained Optimization}
\label{ssec:pdec}

The regularizations imposed in the two previous 
subsections involed the imposition of constraints
on the input $u$ to a PDE model and the resulting
minimizations were expressed in terms of $u$ alone.
For at least two reasons it is sometimes of
interest to formulate the minimization problem
simultaneously over the input variable $u$, together
with the solution of the PDE $p=G(u) \in P$: firstly
computational algorithms which work to find $(p,u)$
in $P \times X$
can be more effective than working entirely in terms
of $u \in X$; and secondly
regularization constraints may be imposed on the
variable $p$ as well as on $u.$ If $J: P \times X \to \bbR$
then this leads to constrained minimization problems
of the form
\begin{equation}
\label{eq:cmin}
\min_{(p,u) \in P \times X} J(p,u):\,\,p=G(u),\,c(p,u) \in
{\cal K}
\end{equation}
where ${\cal K}$ denotes the constraints imposed
on both the input $u$ and on the output $p$ from the
PDE model. Typically the observation operator $\cG:X
\to \bbR^N$
is found from $G$ and then the information
in $\Phi$ can be built into
the definition of $J$.

\begin{example} Consider the running example from
Section \ref{ssec:g} and assume that the observational
noise $\xi \sim N(0,\gamma^2 I).$ Define
$$J(p,u)=\frac{1}{2\gamma^2} \sum_{j=1}^N |y-\ell_j(p)|^2+\frac{\lambda_1}{2} \|u\|_{H^s}^2+\frac{\lambda_2}{2}\|p\|_{P}^2$$
for some $s>d/2.$ 
Choosing $\lambda_1=\lambda$ and $\lambda_2=0$,
together with  $c(p,u)=(p,u)$ and ${\cal K}=P \times X$
we obtain from \eqref{eq:cmin} the minimization
from Example \ref{ex:g2} in the case $\Gamma=\gamma^2 I.$
Choosing $\lambda_1=\lambda_2=0$, $c(p,u)=(p,u)$ and ${\cal K}=
P \times \Ea$ from Example \ref{ex:g99} we recover
that example. Choosing $\lambda_2 \ne 0$ and/or choosing
the constraint set ${\cal K}$ to impose constraints on $p$
leads to minimization in which the output $p$ of the PDE model
is constrained as well as the input $u$ that we are trying
to estimate.
\noindent\qed\end{example}

\subsection{Bayesian Regularization}
\label{ssec:reg3}

The preceding regularization approaches have
a nice mathematical structure and form a natural
approach to the inverse problem when a unique solution
is to be expected. But in many cases it may be
interesting or important to find a large class
of solutions, and to give relative weights to their
importance. This allows, in particular, for
predictions which quantify uncertainty. 
The Bayesian approach to regularization
does this by adopting a probabilistic framework in
which the solution to the inverse problem is
a probability measure on $X$, rather than a single
element of $X$.

We think of $(u, \, y) \in X \times \R^N$ as a random 
variable. Our goal is to find the distribution of 
$u$ given $y$, often denoted by $u|y$. We define 
the joint distribution of $(u, \, y)$ as follows.
We assume that $u$ and $\xi$ appearing in 
\eqref{eq:data} are indepenent mean zero
Gaussian random variables, 
supported on $X$ and $\bbR^N$ respectively,
with covariance operator $\frac{1}{\lambda}\cC$ 
and covariance matrix $\Gamma$ respectively.
By equation \eqref{eq:data}, the distribution 
of $y$ given $u$, denoted $y|u$, is 
Gaussian $N (\cG(u), \Gamma)$. The measure
$\mu_0=N(0,\frac{1}{ \lambda}\cC)$ is known as the 
{\em prior} measure. 
It is most natural to define the measure $\mu_0$
on a Hilbert space $H \supseteq X$. 
Under suitable conditions on $\cC$, we have $\mu_0(X)=1$. This means that under the measure $\mu_0$, $u \in X$ almost surely so that $\cG(u)$ is well-defined, almost surely. If $\mu_0(X) = 1$, it follows that the Hilbert space $E$ 
with norm $\| \cdot \|_E = \|\cC^{-1/2}\cdot \|_{H}$ 
is compactly embedded into $X.$
The space $E$ is known as the Cameron-Martin space.
In the infinite dimensional setting, functions drawn
from $\mu_0$ are almost surely not in the Cameron-Martin
space.
See \cite{bog98,Lif95} for detailed discussion
of Gaussian measures on infinite dimensional spaces.

When solving the inverse problem, the
aim is to find the posterior measure
$\mu^y(du)=\bP(du|y),$ and to obtain information 
about likely candidate solutions to the
inverse problem from it.
Informal application of Bayes' theorem gives 
\be
\bP (u| y) \propto \bP (y|u)\mu_0(u).
\label{eq:bayes}
\ene
The probability density function for
$\bP(y|u)$ is, using the property of
Gaussians, proportional to
$$\exp\bigl(-\frac12|y-\cG(u)|_{\Gamma}^2\bigr)=
\exp \bigl(- \Phi(u)\bigr).$$
The infinite dimensional analogue of this result
is to show that $\mu^y$ is absolutely continuous
with respect to $\mu_0$ with 
Radon-Nikodym derivative relating
posterior to prior as follows:
\be\label{e:radon}
\frac{d \mu^y }{d \mu_0} (u)=\frac{1}{Z} 
\exp \bigl(- \Phi(u)\bigr).
\ene
Here $\Phi(u)$ is given by~\eqref{eq:phi}
and $Z=\int_{X} \exp \bigl(- \Phi(u)\bigr) \mu_0(du).$
The meaning of the formula
\eqref{e:radon} is that expectations under the
posterior measure $\mu^y$ can be rewritten
as weighted expectations with respect to the prior:
for a function $\cF$ on $X$ we may write
$$\int_{X} \cF(u)\mu^y(du)=\int_{X} \frac{1}{Z}\exp \bigl(- \Phi(u)\bigr) \cF(u)\mu_0(du).$$

\begin{theorem} 
(\cite{CDRS08})
Assume that $\mu_0(X)=1$.
Then $\mu^y$ is absolutely continuous with 
respect to $\mu_0$ with Radon-Nikodym derivative given by~\eqref{e:radon}.
Furthermore the measure $\mu^y$ is locally Lipschitz
in the data $y$ with respect to the Hellinger
metric: there is a constant $C=C(r)$, such that, for all $y, \, y'$ with $\max \big\{|y|, \, |y'| \big\} \leq r$,
\begin{equation}\label{e:stability}
d_{\mbox{\tiny{\sc hell}}} (\mu^y , \mu^{y'}) \leq C |y - y'|.
\end{equation}
\end{theorem}

If $\mu, \, \nu$ are probability measures that are absolutely continuous with respect to the probability measure $\rho$, then the Hellinger metric is defined as
\begin{equation*}
d_{\mbox{\tiny{\sc hell}}}(\mu, \nu)^2 = 
\frac{1}{2} \int \left(  \sqrt{\frac{d \mu (u)}{d \rho}} -  \sqrt{\frac{d \nu (u)}{d \rho}} \right)^2 \, \rho (d u).
\end{equation*}
For any function of $u$ which is square integrable
with respect to both $\mu$ and $\nu$ it may be shown
that the difference in expectations of that function,
under $\mu$ and under $\nu$, is bounded above
by the Hellinger distance. In particular, this theorem
shows that the posterior mean and covariance operators
corresponding to data sets $y$ and $y'$
are ${\cal O}(|y-y'|)$ apart.

The choice of prior $\mu_0$, relates directly to
the regularization of the inverse problem. To see
this we note that since the operator $\cC$ is necessarily
positive and self-adjoint we may write down the
complete orthonormal system
\begin{equation}
\label{eq:ONB}
\frac{1}{\lambda}\cC\phi_m=\sigma_m^2 \phi_m,
\quad m \in \bbZ^+, \quad \quad \lim_{m \to \infty} \sigma_m = 0.
\end{equation}
Then $u \sim \mu_0$ can be written via the 
Karhunen-Lo\`eve expansion as
\begin{equation}
\label{eq:KL}
u(x)=\sum_{m \in \bbZ^+} \sigma_m \eta_m \phi_m(x)
\end{equation}
where the $\eta_m$ form an i.i.d. sequence of unit
Gaussian random variables. We may regularize the inverse problem by modifying the decay rate of $\sigma_m$. For example, choosing $\sigma_m=0$
for $m \notin {\cal M}$, where ${\cal M} \subset \bbZ^+$
has finite cardinality restricts the solution of
the inverse problem to a finite dimensional set, and
is hence a regularization. More generally, the rate
of decay of the $\sigma_m$ (which are necessarily summable
as $\cC$ is trace class) will effect the almost sure
regularity properties of functions drawn from $\mu_0$
and, by absolute continuity of $\mu^y$ with respect to
$\mu_0$, of functions drawn from $\mu^y.$ 

In the case that $X$ is a subset of $H = L^2(D)$ with $D \subset \bbR^d$,
the operator ${\cal C}$ may be identified with an integral operator:
$$\frac{1}{\lambda}(\cC \phi)(x_1)=\int_{D}
c(x_1,x_2)\phi(x_2)dx_2$$ for some kernel $c(x_1,x_2)$. The regularity of $c(x_1,x_2)$ determines the decay rate of $\sigma_m$ \cite{Kon86}. If $\cC = (- \Delta)^{\alpha}$ then the corresponding measure $\mu_0$ has the property that samples are almost surely in the Sobolev space $H^s$ and in the H\"older space $C^s$ for all $s<\alpha-\frac{d}{2}$ (see \cite{DPZ92} for more details). In particular, if $\alpha > d/2$, then $\mu(X) = 1$ when $X = L^\infty(D)$.

Priors which charge functions with a multiscale
character can be built in this Gaussian context. One
natural way to do this is to choose ${\cal M}$ as
above so that it contains two distinct sets of functions
varying on length scales of ${\cal O}(1)$ and
${\cal O}(\epsilon)$ respectively. A second natural
way is to choose a covariance function $c=c^{\epsilon}$
which has two scales. 

The formula \eqref{e:radon}
shows quite clearly  how regularization
works in the Bayesian context: the main contribution
to the expectation will come from places where $\Phi$
is close to its minimum value and where $\mu_0$ is concentrated;
thus minimizing $\Phi$ is important, but this
minimization is regularized through 
the properties of the measure $\mu_0$. We now develop this intuitive concept further
by linking the Bayesian approach to Tikhonov
regularization and the functional $I$ given by
\eqref{eq:IT}.

Given $z \in E$ and $\delta \ll 1$
define the small ball probability 
$$J^{\delta}(z)=\bP^{\mu^y}\bigl(\|u-z\|_{X}<\delta\bigr).$$
Note that this ball is in $X$ but centred at a point
$z \in E$, with $E$ (the Cameron-Martin space) compact in $X$. 
It is natural to ask where $J^{\delta}(z)$ is maximized
as a function of $z$ and placing $z$ in $E$ allows
us to answer this question. Furthermore
we then see a connection between
the Bayesian approach and the Tikhonov approach
to regularization. The next theorem
shows that small balls centred at
minimizers of \eqref{eq:IT} will have
maximal relative probability under the Bayesian
posterior measure, in the small ball limit
$\delta \to 0.$

\begin{theorem} (\cite{DS10})
\label{t:sb}
Assume that $\mu_0(X)=1$. Then
$$\lim_{\delta \to 0} \frac{J^{\delta}(z_1)}
{J^{\delta}(z_2)}=\exp\bigl(I(z_2)-I(z_1)\bigr).$$
\end{theorem}

In the Bayesian context the solution of the
Tikhonov regularized problem is known as
the Maximum A Posteriori
estimator (MAP estimator) \cite{Berger, Fitzp91}.

\section{Large Data Limits}
\label{sec:invm}

In the previous section we showed how regularization
plays a significant role in the solution of inverse
problems. Choosing the correct regularization is
part of the overall modelling scenario in which
the inverse problem is embedded, as we demonstrated
in the running example of Section \ref{ssec:g}.
In some situations it may be suitable to look for
the solution of the inverse problem over a small
finite set of parameters, whilst in others it may
be desirable to look over a larger, even infinite
dimensional set, in which oscillations are captured.

This section is devoted
entirely to inverse problems where a single
scalar parameter is sought and we study whether
or not this parameter is correctly identified
when a large amount of noisy data is available.
The development is tied specifically to the
running example, namely the PDE \eqref{eq:pde1}. For a fixed permeability coefficient generating the data, Fitzpatrick has also studied the consistency and asymptotic normality of maximum likelihood estimates in the large data limit \cite{Fitzp91}. Related work on parameter estimation in the context stochastic differential equations (SDEs) may
be found in \cite{PavlSt06,APavlSt09}.

\subsection{The Statistical Model}
\label{ssec:ldata}

We consider the problem of estimating a single scalar
parameter $u \in \bbR$ in the elliptic PDE 
\begin{align}
\begin{split}
\label{eq:pde-1}
\nabla \cdot v&=f,\quad x \in D,\\
p&=0, \quad x \in \partial D,\\
v&=-\exp(u)A \nabla p
\end{split}
\end{align}
where $D \subset \bbR^d$ is bounded and open, and $f \in H^{-1}$ as well as the constant symmetric matrix $A$ are assumed to be known. We let $G:\bbR \to H^1_0(D)$ be defined by $G(u)=p.$ Then using the same linear functionals as in the running example from Section \ref{ssec:g} we may construct the observation operator $\cG:\bbR \to \bbR^N$ defined by $\cG(u)_j=\ell_j(G(u)).$ Our aim is to solve the inverse problem of determining $u$ given $y$ satisfying \eqref{eq:data}. For simplicity we assume that $\xi \sim N(0,\gamma^2 I)$ which implies that the observational noise on each linear functional is i.i.d. $N(0,\gamma^2).$ Since $u$ is finite dimensional we will simply minimize $\Phi$ given by \eqref{eq:phi}: no further regularization is needed because $u$ is already
finite dimensional.

Notice that the solution $p$ of \eqref{eq:pde-1} is linear in $\exp(-u)$ and that we may write $G(u)=\exp(-u)\ps$ where $\ps$ solves 
\begin{align}
\begin{split}
\label{eq:pde-2}
\nabla \cdot v&=f,\quad x \in D,\\
\ps&=0, \quad x \in \partial D.\\
v&= -A \nabla \ps
\end{split}
\end{align}

Note that
$\cG(u)_j=\exp(-u)\ell_j(\ps)$ 
so that the least squares functional \eqref{eq:phi} has
the form
$$\Phi(u)=
\frac{1}{2\gamma^2}\sum_{j=1}^N
|y_j-{\cal G}_j(u)|^2
=\frac{1}{2\gamma^2}\sum_{j=1}^N
|y_j-\exp(-u)\ell_j(\ps)|^2.$$
It is straightforward to see that $\Phi$ has
a unique minimizer $\bu$ satisfying
\begin{equation}
\exp(-\bu)=\frac{\sum_{j=1}^{N} y_j \ell_j(\ps)}
{\sum_{j=1}^{N} \ell_j(\ps)^2}.
\label{eq:est}
\end{equation}
It is now natural to ask whether, for large $N$,
the estimate $\bu$ is close to the desired value
of the parameter. We study two situations: the
first where the data is generated by the model
which is used to fit the data; and the second where
the data is generated by a multiscale model whose
homogenized limit gives the model which is used to 
fit the data.

\subsection{Data From the Homogenized Model} 
\label{ssec:homd}

We define $\po=\exp(-\uo)\ps$ so that $\po$ solves 
\eqref{eq:pde-1} with $u=u_0.$

\begin{assumption}
\label{ass:data1}
We assume that the data $y$ is generated from
noisy observations generated by the statistical model:
$$y_j=\ell_j(\po)+\xi_j$$
where $\{\xi_j\}$ form an i.i.d. sequence of random
variables distributed as $N(0,\gamma^2).$
\end{assumption}

\begin{theorem}
\label{t:con1}
Let Assumptions \ref{ass:data1} hold and assume that
$\liminf_{N \to \infty} \frac{1}{N}\sum_{j=1}^N \ell_j(\ps)^2 \geq L > 0$ as $N \to \infty.$ Then
$\xi$-almost surely 
$$\lim_{N \to \infty}
|\exp(-\bu)-\exp(-u_0)|=0.$$
\end{theorem}

\begin{proof}
Substituting the assumed expression for
the data from Assumption \ref{ass:data1}
into the formula \eqref{eq:est} gives
$$\exp(-\bu)=\exp(-\uo)+I_1$$
where
$$I_1=\frac{\frac{1}{N}\sum_{j=1}^N \xi_j\ell_j(\ps)}{\frac{1}{N}\sum_{j=1}^N \ell_j(\ps)^2}.$$
Therefore,
\begin{equation}
\bbE [I_1^2] = \frac{\gamma^2}{ \sum_{j=1}^N \ell_j(p^*)^2 } \leq  \frac{2\gamma^2}{NL}
\end{equation}
for $N$ sufficiently large. Since
$I_1$ is Gaussian we deduce that $\bbE I_1^{2p}={\cal O}(N^{-p})$ as $N \to \infty$. Application of the Borel-Cantelli lemma shows that
$I_1$ converges almost surely to zero as $N \to \infty$.
\qed
\end{proof}

This shows that, in the large data limit, random
observational error may be averaged out and the
true value of the parameter recovered, in the
idealized scenario where the data is taken from
the statistical model used to identify the
parameter. The condition that $L > 0$ prevents additional observation noise from overwhelming the information obtained from additional measurements as $N \to \infty$. It is a simple explicit example of
what is known as {\em posterior consistency}
\cite{bickel} in the theory of statistics.

\subsection{Data From the Multiscale Model} 
\label{ssec:muld}

In practice, of course, real data does not come
from the statistical model used to estimate
parameters. In order to probe the effect that
this can have on posterior consistency we
study the situation where the data is taken from
a multiscale model whose homogenized limit falls
within the class used in the statistical model
to estimate parameters.
Again we define $\po=\exp(-\uo)\ps$
and we now define $\pe$ to solve \eqref{eq:pde22}
with $K^{\epsilon}$ chosen so that the homogenized
coefficient associated with this family is 
$K_0=\exp(u_0)A.$

\begin{assumption}
\label{ass:data}
We assume that the data $y$ is generated from
noisy observations of a multiscale model:
$$y_j=\ell_j(\pe)+\xi_j$$
with $\pe$ as above and 
the $\{\xi_j\}$ an i.i.d. sequence of random
variables distributed as $N(0,\gamma^2).$
\end{assumption}

\begin{theorem}
\label{t:con}
Let Assumptions \ref{ass:data} hold and assume that
that the linear functionals $\ell_j$ are chosen so that
\begin{equation}\label{Nepsdoublelim}
\lim_{\epsilon \to 0} \limsup_{N \to \infty} \frac{1}{N}\sum_{j=1}^N |\ell_j(\pe-\po)|^2 = 0
\end{equation}
and $\liminf_{N \to \infty} \frac{1}{N}\sum_{j=1}^N \ell_j(\ps)^2 \geq L > 0$ as $N \to \infty.$ Then
$\xi-$ almost surely 
$$\lim_{\epsilon \to 0} \lim_{N \to \infty}
|\exp(-\bu)-\exp(-u_0)|=0.$$
\end{theorem}

\begin{proof}
Notice that the solution of the homogenized equation
is $\po=\exp(-u_0)\ps.$ We write
\begin{align*}
y_j&=\ell_j(\po)+\ell_j(\pe-\po)+\xi_j\\
&=\exp(-u_0)\ell_j(\ps)+\ell_j(\pe-\po)+\xi_j.
\end{align*}
Substituting this into the formula \eqref{eq:est}
gives
$$\exp(-\bu)=\exp(-\uo)+I_1+I_2^\epsilon$$
where $I_1$ is as defined in the proof of Theorem
\ref{t:con1} and is independent of $\epsilon$, and
$$I_2^\epsilon=
\frac{\sum_{j=1}^N \ell_j(\pe-\po)\ell_j(\ps)}{\sum_{j=1}^N \ell_j(\ps)^2}.$$
The Cauchy-Schwarz inequality gives
\[
|I_2^\epsilon| \le \frac{\Bigl(\sum_{j=1}^N |\ell_j(\pe-\po)|^2\Bigr)^{1/2}}{\left(  \sum_{j=1}^N \ell_j(\ps)^2 \right)^{1/2} } \leq \Bigl(\frac{2}{NL}\sum_{j=1}^N |\ell_j(\pe-\po)|^2\Bigr)^{1/2}
\] 
for $N$ sufficiently large. As in the proof of Theorem \ref{t:con1} we have, $\xi$-almost
surely,
$$\lim_{N \to 0}|\exp(-\bu)-\exp(-\uo)-I_2^\epsilon|=0.$$
From this and (\ref{Nepsdoublelim}) the desired result now follows.
\qed
\end{proof}

The assumption (\ref{Nepsdoublelim}) encodes the idea that, for small $\epsilon$,
the linear functionals used in the observation
process return nearby values when
applied to the solution $\pe$ of the multiscale model
or to the solution $\po$ of the homogenized equation. In particular, Corollary \ref{cor:homcon} implies that if $\{ \ell_j(p) \}_{j=1}^\infty$ is a family of bounded linear functionals on $L^2(D)$, uniformly bounded in $j$, then (\ref{Nepsdoublelim}) will hold. On the other hand, we may choose linear functionals that are bounded as functionals on $H^1(D)$ yet unbounded on $L^2(D)$. In this case Theorem \ref{t:homcon} shows that (\ref{Nepsdoublelim}) may not hold and the correct
homogenized coefficient may not be recovered, even in
the large data limit.  An analogous phenomenon
occurs in inference for SDEs where if the observations
of a multiscale diffusion are too frequent 
(relative to the fast scale) then the correct homogenized
coefficients are not recovered \cite{PavlSt06,APavlSt09}.

\section{Exploiting Multiscale Properties Within Inverse
Estimation}
\label{sec:invm2}

In this section we describe how ideas
from homogenization theory can be
used to improve the estimation of parameters
in homogenized models. We consider a regime where the unknown parameter has small-scale fluctuations that may be characterized as random. In this case, if we attempt to recover the homogenized parameter the error $\xi$ appearing in \eqref{eq:data} is affected by the model mismatch. This is because the simplified, low-dimensional parameter used to fit the data is different from the true unknown coefficient. So, even when there is no observational noise, the error $\xi$ has a statistical structure. Nevertheless, homogenization theory predicts that this discrepancy between $G(u)$ and $y$ associated with model mismatch will have a universal statistical structure which can be exploited in the inverse problem, as we now describe. 

The specific ideas described here were developed by Nolen and Papanicolaou 
in \cite{NP09} for one dimensional elliptic problems, including the groundwater flow problem that we study here. Bal and Ren \cite{BalRen09} have employed similar ideas in the study of Sturm-Liouville problems with unknown potential. We begin by describing in Section \ref{ssec:rhf} the homogenization and fluctuation theory for the case that the (scalar) permeability $k(x)$ is random. Then,
in Section \ref{ssec:ies} we show how these
ideas can be used to develop an improved
estimator for the homogenized permeability coefficient. We conclude with numerical
results in Section \ref{ssec:num}.

\subsection{The Model}
\label{ssec:rhf}
In this section we will present the approach of \cite{NP09} in the simplest possible setting.
We consider the two-point boundary value problem
\begin{subequations}
\label{e:pde_nol_pap2}
\begin{eqnarray}
- \frac{d}{d x} \left( \exp(u(x)) \frac{d p}{d x} \right)  &=& f(x), \quad x \in [-1,1], \\
p(-1) =p(1) &=& 0.
\end{eqnarray}
\end{subequations}
This is, of course, \eqref{eq:pde1} in the
one-dimensional setting $d=1.$

It is assumed that the coefficient $k(x) = \exp(u(x))$ is a single
realization of a stationary, ergodic and mixing random
field $k(x,\omega).$ Furthermore it is assumed that
$k^{-1}$ can be decomposed into a slowly varying non-random component,
together with a random, rapidly oscillating component:
\be\label{e:coeff_model}
\frac{1}{k(x, \omega)} = \frac{1}{k_0(x)} + \sigma \mu \left(\frac{x}{\eps}, \omega  \right),
\ene
where $\mu (x,\omega)$ is a stationary, mean zero random field with covariance
\begin{equation*}
R(x) = \Ex (\mu(x+y) \mu(y)).
\end{equation*}
We assume that $R(0) = 1$ and $\int_{\bbR} R(x) \, dx =1$. Thus, $\sigma^2$ and $\eps$ are the (given) variance and correlation length of the fluctuations. We are interested in the
case where $\eps \ll 1$ so that the random fluctuations are
rapid.

The solution $p = p_\epsilon(x,\omega)$ of~\eqref{e:pde_nol_pap2} depends on $\epsilon > 0$ and on the realization of $k(x,\omega)$. However, in the limit as $\eps \rightarrow 0$, $p_\epsilon$ coverges to $p_0(x)$ which is the solution of the homogenized Dirichlet problem 
\begin{subequations}
\label{e:pde_homog}
\begin{eqnarray}
- \frac{d}{dx} \left( k_0(x) \frac{d}{dx} p_0 \right) &=& f(x), \quad x \in [-1,1],
\\
\po(-1) =\po(1) &=& 0.
\end{eqnarray}
\end{subequations}
Observe that the homogenized coefficient is the harmonic mean of $k$: $k_0(x) = \Ex[k^{-1}]^{-1}$. Moreover,  in the limit as $\eps \rightarrow 0$, the solution $p_\epsilon$ has Gaussian fluctuations about its asymptotic limit \cite{BGMP08}. Specifically, one can prove that
\be\label{e:lim_thm}
\frac{p_\epsilon(x, \omega) - \po(x)}{\eps^{1/2}} \rightarrow \sigma \int_D Q(x,y ; k_0) v_0(y;k_0) \, d W_y (\omega)
\ene
in distribution as $\eps \rightarrow 0$, where $W_y (\omega)$ is a Brownian random field, which is a Gaussian process. Here $v_0(x;k_0) = k_0(x) p_0(x)$, and the kernel $Q(x,y ; k_0)$ is then 
related to the Green's function for the one dimensional system:
\begin{eqnarray*}
\left( \begin{array}{c} p_x \\ v_x \end{array} \right) - \left( \begin{array}{cc} 0 & 1/k_0(x) \\ 0 & 0 \end{array} \right)\left( \begin{array}{c} p \\ v \end{array} \right) = \left( \begin{array}{c} g_1 \\ g_2 \end{array} \right).
\end{eqnarray*}
If the $2\times2$ Green's matrix for this system is $G(x,y;k_0):D \times D \to \R^2 \otimes \R^2$, then $Q(x,y;k_0) = G_{1,1}(x,y;k_0)$. The important point here is that the integral 
\[
I(x,\omega) = \sigma \int_D Q(x,y ; k_0) v_0(y;k_0) \, d W_y (\omega)
\]
which appears on the right side of (\ref{e:lim_thm}) is a centered Gaussian random variable with covariance
\[
\Ex[I(x)I(z)] = \sigma^2 \int_D Q(x,y;k_0) v_0(y;k_0)^2 Q(y,z;k_0) \,dy.
\]
This covariance depends on  $k_0.$ 
The asymptotic theory given by the limit theorem
\eqref{e:lim_thm}
gives us a good approximation of the statistics of $p_\epsilon(x,\omega)$ even when there is no observation noise, and
shows that the fluctuations depend on $k_0.$ 
In this simple case presented here, $Q$ can be computed explicitly. In other cases, it can be computed numerically; see \cite{NP09} for more details.

\subsection{Enhanced Estimation}
\label{ssec:ies}

We now show how this asymptotic theory can be
used to enhance estimation of the homogenized parameter $k_0(x)$. The inverse problem is to identify the parameter
$k_0(x)$ in the model 
\begin{subequations}
\label{e:homogr}
\begin{eqnarray}
- \frac{d}{dx} \left( k_0(x) \frac{d}{dx} p_0 \right) &=& f(x), \quad x \in [-1,1],\\
p_0(-1) =p_0(1) &=& 0.
\end{eqnarray}
\end{subequations}
We take the viewpoint that the data actually come from observations of $p_\epsilon(x,\omega)$, which is the solution of the multiscale model \eqref{e:pde_nol_pap2} with $k(x,\omega)$ given by \eqref{e:coeff_model}, so there is a discrepancy between the model used to fit the data and the true model which generates the data. Now the outstanding modelling issue is the choice of statistical model for the error $\xi$ in~\eqref{eq:data}.

Suppose we make noisy observations of $p_\epsilon(x_j)$ at points $\{ x_j \}_{j=1}^N$ distributed throughout the domain. Then the measurements are
\[
y_j = p_\epsilon(x_j,\omega) + \xi_j, \quad j=1,\dots,N
\]
where $\xi_j \sim N(0,\gamma^2)$ are mutually independent, representing observation noise. The limit (\ref{e:lim_thm}) we have just described tells us that for $\epsilon$ small, these measurements are approximated well by
\[
y_j \approx p_0(x_j) + \xi_j',
\]
where $\{\xi_j'\}_{j=1}^N$ are Gaussian random variables with mean zero
and covariance
\begin{equation}
C_{j,\ell}(k_0,\epsilon) = \Ex[\xi_j' \xi_\ell'] =  \gamma^2 \delta_{j,\ell}  + \epsilon \sigma^2 \int_D Q(x_j,y;k_0) v_0(y;k_0)^2 Q(x_\ell,z;k_0) \,dy \label{covobshom}
\end{equation}
Therefore, we model the observations as
\[
y_j \approx \cG(k_0) + \xi_j', \quad j=1,\dots,N
\]
where $\cG(k_0) = p_0(x_j;k_0)$ with $p_0$ being the solution of $(\ref{e:homogr})$. The modified statistical error $\xi'$ has two components.  The first term $\gamma^2 \delta_{j,\ell} $ is due to observation error. The second term comes from the asymptotic theory and is associated with the random microstructure in the true parameter $k(x,\omega)$. Of course, if $\epsilon$ is very small, relative to $\gamma^2$, then the observation noise dominates (\ref{covobshom}). In this case, the observations of $p_\epsilon$ may be very close to observations of the homogenized solution $p_0$, and we might simply assume that $\xi' \sim N(0,\gamma^2 I)$, ignoring the error associated with the model mismatch.  On the other hand, if $\gamma^2$ is small relative to $\epsilon$ then the statistical error $\xi'$ is dominated by the model mismatch. In this case, homogenization theory gives us an asymptotic approximation of the true covariance structure of $\xi'$, which is quite different from $N(0,\gamma^2 I)$. See \cite{NP09} for a discussion of some properties of the covariance matrix $C(k_0,\epsilon)$.

Using the covariance (\ref{covobshom}), we make the approximation 
\begin{equation*}
\bbP (y|k_0) \approx \frac{1}{\sqrt{2\pi|C(k_0;\epsilon)|}}\exp \Bigl(-\frac{1}{2} \bigl(y - \cG(k_0)\bigr)^T C(k_0;\epsilon)^{-1} \bigl(y- \cG(k_0)\bigr) \Bigr),
\end{equation*}
where $|\cdot|$ denotes the
determinant. The parameter $k_0(x)$ is a function, in general, and we may place a Gaussian prior $\mu_0$ on $u_0(x) = \log k_0(x)$. Application of Bayes' theorem  
\eqref{eq:bayes} (with $k_0$ replacing $u$) gives that
\be\nonumber
\bbP(k_0|y) \propto \frac{1}{\sqrt{2\pi|C(k_0;\epsilon)|}}\exp \Bigl(-\frac{1}{2} \bigl(y - \cG(k_0)\bigr)^T C(k_0;\epsilon)^{-1} \bigl(y- \cG(k_0)\bigr) \Bigr)\mu_0(\log k_0) \label{mapP}
\ene
where the constant of proportionality is independent
of $k_0$. The maximum a posteriori estimator (MAP) is then found as the function $k_0(x)$
which maximizes $\bbP(k_0|y)$ which is the same as minimizing $I(k_0)=-\ln\bigl(\bbP(k_0|y)\bigr).$
The key contribution of homogenization theory is
to correctly identify the noise structure
which has covariance $C(k_0;\epsilon)$ depending on $k_0(x)$, the parameter to be estimated.

\subsection{Numerical Results}
In this section we demonstrate the results of a numerical computation that show some advantage to using the homogenization theory as we have just described. Given noisy observations of $p_\epsilon(x_j)$ we may compute the MAP estimator $\hat k_1$ using (\ref{mapP}) with covariance $C(k_0;\epsilon)$ given by (\ref{covobshom}):
\begin{equation}\label{k1estdef}
\hat k_{1} = \text{argmax}_{k_0} \; \frac{1}{\sqrt{2\pi|C(k_0;\epsilon)|}}\exp \Bigl(-\frac{1}{2} \bigl(y - \cG(k_0)\bigr)^T C(k_0;\epsilon)^{-1} \bigl(y- \cG(k_0)\bigr)\Bigr)\mu_0(\log k_0), 
\end{equation}
On the other hand, we might ignore the effect of the random microstructure and simply use $C = \gamma^2 I$, accounting only for observation noise:
\begin{equation}\label{k2estdef}
\hat k_{2} = \text{argmax}_{k_0} \; \frac{1}{\sqrt{2\pi|\gamma^2 I|}}\exp \Bigl(-\frac{1}{2} \gamma^{-2} | y - \cG(k_0)|^2  \Bigr)\mu_0(\log k_0).
\end{equation}
Both estimates $\hat k_1$ and $\hat k_2$ are random variables, depending on the random data observed, but we should hope that $\hat k_1$ gives us a better approximation of $k_0$, since it makes use of the true covariance (\ref{covobshom}).  Indeed for simple linear statistical models, it is easy to see that an efficient estimator, which realizes the theoretically optimal variance given by the Cram\'er-Rao lower bound, may be obtained by using the true covariance of the data; however, using the incorrect covariance may lead to an estimate with significantly higher variance than the theoretical optimum. See \cite{NP09} for more discussion of this point. The present setting is highly nonlinear and the variance of the estimates $\hat k_1$ and $\hat k_2$ cannot be computed explicitly, since $C(k_0,\epsilon)$ depends on $k_0$ in a nonlinear way through solution of the PDE. Nevertheless the numerical results are consistent with the expectation that approximation of the true covariance (through homogenization theory) yields a MAP estimator that has smaller variance, relative to the estimate that makes no use of the homogenization theory (see Figure \ref{fig:varplot}).

\begin{figure}[htbp] 
\centering
\includegraphics[width=350pt]{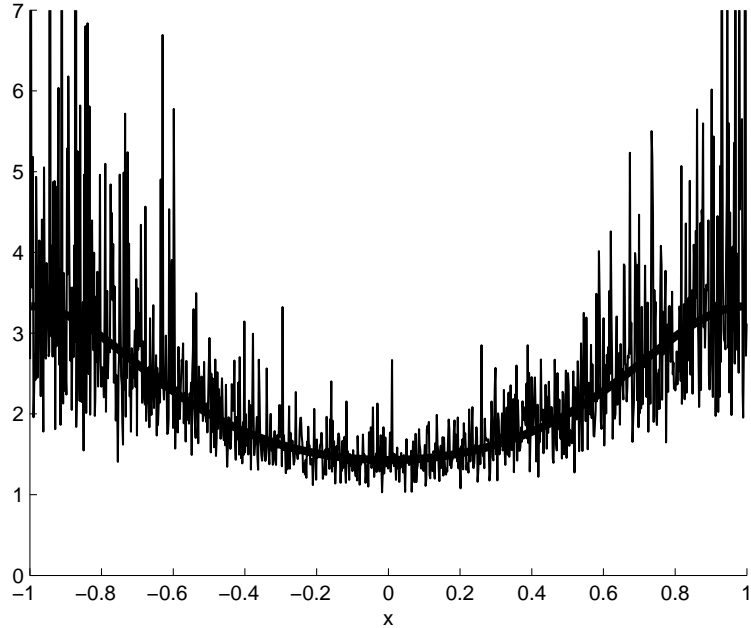}
\caption{The thin erratic curve is one realization of the true coefficient $k^\epsilon(x,\omega)$. The thick curve is the slowly-varying harmonic mean $k_0(x)$. This realization was used to generate the data.} \label{fig:ktrue}
\end{figure}

In Figure \ref{fig:ktrue} we show one realization of the true coefficient $k(x,\omega)$ which was used to generate the data. The highly-oscillatory graph represents the true coefficient $k(x,\omega)$ with variation on many scales. The slowly-varying harmonic mean $k_0(x)$ also is displayed here as the thick curve; this function $k_0$ is what we attempt to estimate. The data was generated as follows.  Using one realization of $k(x,\omega)$ and given forcing $f$, we solve the Dirichlet boundary value problem (\ref{e:pde_nol_pap2}). The observation data involves point-wise evaluation of $p^\epsilon(x_j)$ at points $\{ x_j \}_{j=1}^N$ spaced uniformly across the domain, plus independent observation noise $N(0,\gamma^2)$ at each point of observation. Using this data, we compute estimates $\hat k_1$ and $\hat k_2$ by minimizing (\ref{k1estdef}) and (\ref{k2estdef}), respectively. For the computation shown here, the function $k_0(x)$ is parameterized by the first three coefficients in a Fourier series expansion. So, computing $\hat k_1$ and $\hat k_2$ involves an optimization in $\mathbb{R}^3$. To evaluate $\bbP(k_0 | y)$ at each step in the minimization algorithm, we must solve the forward problem (\ref{e:homogr}) with the current estimate of $k_0$, and in the case of $\hat k_1$ we must also compute $C(k_0,\epsilon)$.  See \cite{NP09} for more details about this computation. 

Figure \ref{fig:kplot} compares the estimate $\hat k_1(x)$ with the true function $k_0(x)$. Since the estimate $\hat k_1(x)$ is a random function, we performed the experiment many times (generating new $k(x,\omega)$ to compute each estimate $\hat k_1$) and display the results of 100 experiments. The data for $\hat k_2$ is qualitatively similar. Nevertheless, the pointwise variance $Var[\hat k_1(x)]$ is smaller than $Var[\hat k_2(x)]$, as shown in Figure \ref{fig:varplot}. This is consistent with the linear estimation theory for which knowledge of the true data covariance yields an estimate with optimal variance.

\begin{figure}[htbp] 
\centering
\includegraphics[width=350pt]{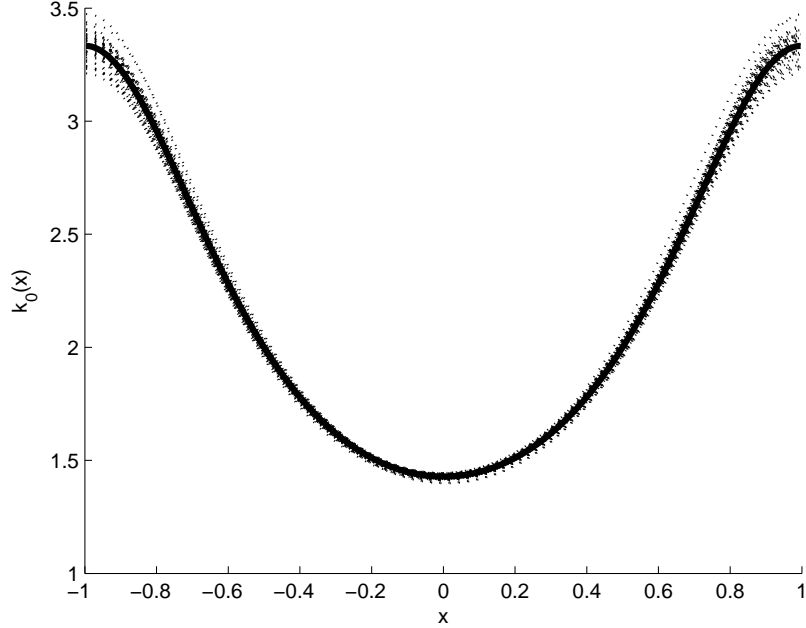}
\caption{The thick curve is the true $k_0$. The dashed series represent 100 independent realizations of the estimate $\hat k_1$.} \label{fig:kplot}
\end{figure}

\begin{figure}[htbp] 
\centering
\includegraphics[width=350pt]{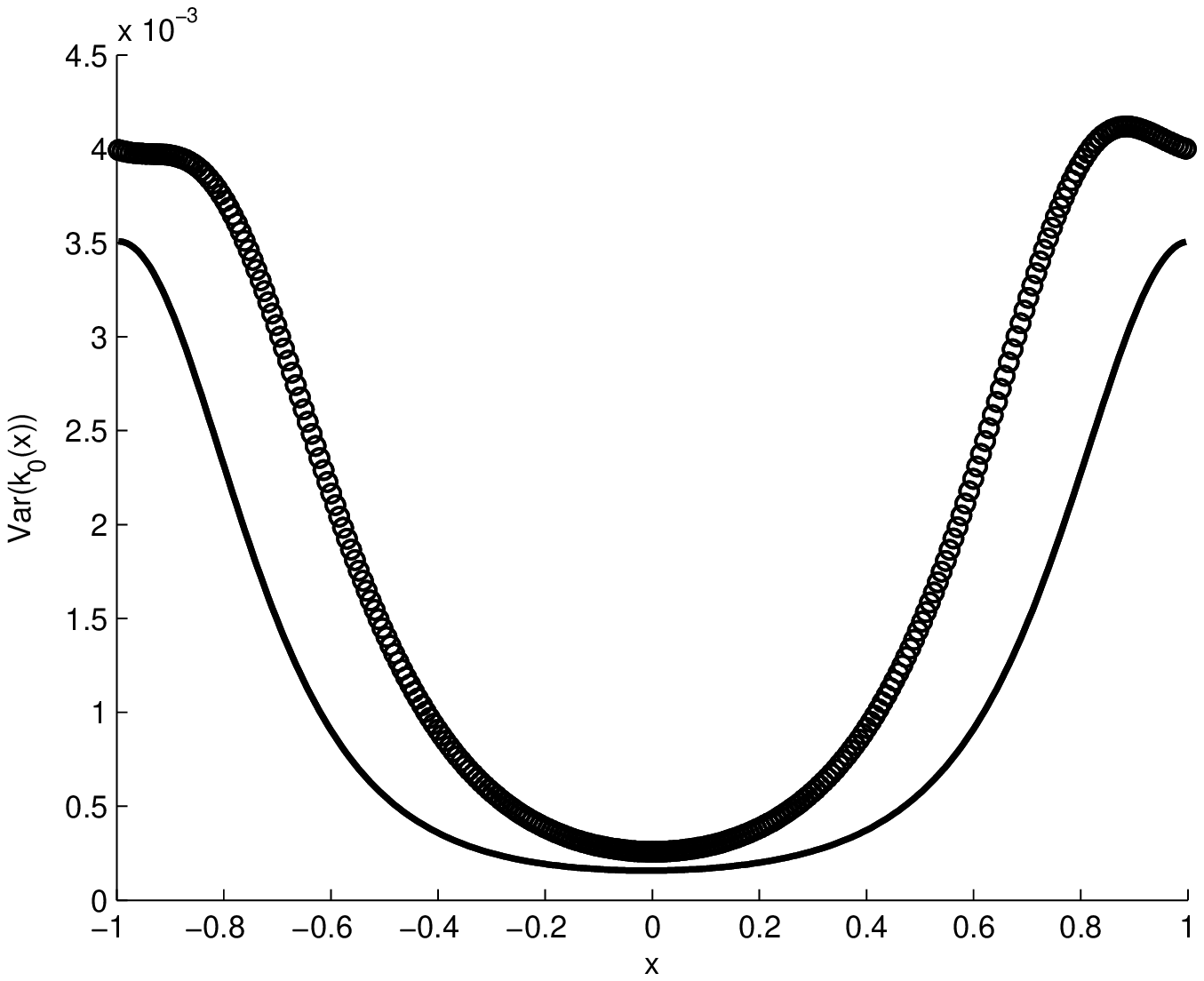}
\caption{The upper series (o) is the empirical variance $Var[ \hat k_2(x)]$. The lower series (-) is $Var[\hat k_1(x)]$. Both quantities were computed using 500 samples.} \label{fig:varplot}
\end{figure}

\label{ssec:num}
\begin{acknowledgement}
The authors thank A. Cliffe and Ch. Schwab
for helpful discussions concerning the groundwater flow model.
\end{acknowledgement}
%
%
%

\section*{Appendix 1}
\addcontentsline{toc}{section}{Appendix}
\label{app:a}

In this Appendix we prove Theorem~\ref{t:avcon} which, recall,
applies in the case 
where (\ref{eq:pde22}b) and (\ref{eq:pde2}b) are replaced
by periodic conditions on $D=(L\bbT)^d$.

\begin{theorem} 
Let $\xe (t)$ and $\xo (t)$ be the solutions to 
equations~\eqref{eq:sde22} and~\eqref{eq:ode22}, with
velocity fields extended from $D=(L\bbT)^d$ to $\bbR^d$
by periodicity, and assume that 
Assumption~\ref{ass:1} holds. 
Assume also that $f \in C^{\infty}(D)$ and that $K(x,y) \in C^{\infty}(D ; C^{\infty}_{\rm{per}}(\T^d))$. Then
$$\lim_{\epsilon \to 0} \bbE \sup_{0 \le t \le T}
\|\xe(t)-\xo(t)\|=0.$$
\end{theorem}

\begin{proof}
To simplify the notation we will set the porosity of the rock to be equal to $1$, $\phi =1$. Recall that $\ve(x)=K^\eps (x) \nabla p^\eps (x).$ Our first observation is that, for
$\pea(x)$ given by \eqref{eq:pap},
\be\label{e:veps}
K^\eps (x) \nabla p^\eps (x) = K^\eps (x) \nabla \pea (x) - \delta^{\eps}(x)
\ene
where
\be\label{e:delta}
\delta^\eps (x) = -K^{\eps}(x)\nabla \Big( p^\eps (x) - \pea(x) \Big) .
\ene
From Assumption~\ref{ass:1} we deduce that
\be\nonumber
\lim_{\eps \rightarrow 0} \|\delta^\eps (x) \|_{L^{\infty}} = 0.
\ene
From the definition of $\pea(x)$ it follows that
\be\nonumber
K^{\eps}(x) \nabla \pea (x) = Q^{\eps}(x) \nabla p_0(x) - \eps \delta^{\eps}_1(x)
\ene
where 
\be\label{e:delta1}
\delta_1^\eps (x) =  -K^\eps (x) \nabla_x p_1(x,x/\eps),
\quad Q^{\eps}(x)=Q(x,x/\eps). 
\ene
From the definition of $p_1$ in \eqref{eq:p1} we see that
\be\nonumber
\|\delta_1^{\eps}(x) \|_{L^{\infty}} \leq C.
\ene
Putting \eqref{e:veps} and~\eqref{e:delta1} together
we see that
$$\ve(x)=  -Q^\eps (x) \nabla p_0(x) + \delta^\eps (x) + \eps \delta_1^\eps (x)
$$
and we see from \eqref{e:delta} and \eqref{e:delta1} that
the perturbations of $\ve(x)$ from $Q^\eps (x) \nabla p_0(x)$
are small; it is thus natural to expect a limit theorem
for $\xe$ solving \eqref{eq:sde22} which is Lagrangian
transport in an appropriately averaged version of  
$Q^\eps (x) \nabla p_0(x).$ Furthermore,
since $Q(x,y)$ is divergence free in the fast $y$
coordinate, by \eqref{eq:Qprop},
it is natural to expect that the
appropriate average is Lebesgue measure.
We now demonstrate that this
is indeed the case.

From~\eqref{eq:sde22} we deduce that
\be
x^\eps(t) = x(0) + \int_0^t \left(- Q^\eps (x) \nabla p_0(x(s)) + \delta^\eps (x(s)) + \eps \delta_1^\eps (x(s)) \right) \, ds + \sqrt{2 \eta_0 \eps} \, W(t).
\label{eq:a7}
\ene
Define now $V(x,y) =- Q(x,y) \nabla \po (x)$ and consider the system of SDEs
\begin{subequations}\label{e:system_sde}
\be
\frac{dx}{d t} = \left( V(x,y) + \de(x) + \eps \dd(x) \right)+\sqrt{2 \eta_0 \eps} \, \frac{d W}{d t},
\ene
\be
\frac{dy}{d t} =  \frac{1}{\eps} \bigl(V(x,y) + \de(x)\bigr) + \dd(x) + \sqrt{\frac{2 \eta_0}{ \eps}} \, \frac{d W}{d t}.
\ene
\end{subequations}
Since $y = x/\eps$ we see that $x(t)$, the solution of~\eqref{e:system_sde} is equal to $\xe(t)$ appearing
in \eqref{eq:a7}. 

The process $\{ x(t), \, y(t) \}$ is Markov with generator
\begin{eqnarray*}
\cL & = & \frac{1}{\eps} \Bigl(\bigl(V(x,y)+\de(x)\bigr) \cdot \nabla_y + \eta_0 \Delta_y \Bigr)\\
&& + \Bigl(\bigl(V(x,y)+\de(x)\bigr) \cdot \nabla_x 
+\dd(x)\cdot \nabla_y
+ \eta_0\nabla_x \cdot \nabla_y 
+ \eta_0\nabla_y \cdot \nabla_x 
\Bigr)\\
&&\quad\quad\quad\quad + \eps \eta_0 \Delta_x+\eps\dd(x)\cdot\nabla_x 
    \\ & =: & \frac{1}{\eps} \bigl(\cL_0+\de(x)\cdot\nabla_y\bigr) + \cL_1 + \eps \cL_2.
\end{eqnarray*}
Consider now the Poisson equation
\be\label{e:poisson_proof}
- \cL_0 \Phi = V(x,y) - \vo (x)
\ene
with (see~\eqref{eq:pde2}(c))
\be\nonumber
\vo (x) = \int_{\T^d} V(x,y) \, dy.
\ene
Equation~\eqref{e:poisson_proof} is posed on $\T^d$ with periodic boundary conditions. Notice that $x$ enters merely as a parameter in this equation. The operator $\cL_0$ is uniformly elliptic on $\T^d$ and the right hand side averages to $0$, hence, by Fredholm's alternative this equation has a solution which is unique, up to constants. We fix this constant by requiring that $\int_{\T^d} \Phi(x,y) \, dy =0$. We define $\Phi^{\eps}(x):=\Phi(x, x/\eps)$ and similarly for ${\cal L}_i\Phi^{\eps}(x)$.
Applying It\^{o}'s formula to $\Phi$ and evaluating
at $y=x/\eps$ we obtain 
\begin{eqnarray*}
d \Phi^{\eps}(x) & = & \frac{1}{\eps} \left(\cL_0 \Phi^\eps
+\de(x)\cdot\nabla_y\Phi \left( x,x/\eps \right)\right) \, dt + \cL_1 \Phi^\eps \, dt + \eps \cL_2 \Phi^\eps \, dt\\
&&\quad\quad\quad\quad +\sqrt{\frac{2\eta_0}{\eps}} \nabla_y \Phi^\eps d W + \sqrt{2 \eta_0 \eps} \nabla_x \Phi^\eps dW
           \\    & = & - \frac{1}{\eps} \left( V \left(x, x/\eps \right) - \vo(x) +\de(x)\cdot\nabla_y\Phi(x,x/\eps) \right) \, dt + \cL_1 \Phi^\eps   + \eps \cL_2 \Phi^\eps \, dt\\
&&\quad\quad\quad\quad\quad\quad\quad\quad +\sqrt{\frac{2 \eta_0}{\eps}} \nabla_y \Phi^\eps d W + \sqrt{2 \eta_0 \eps} \nabla_x \Phi^\eps dW.
\end{eqnarray*}
Consequently,
\begin{align*}
\int_0^t &V\bigl(x(s),y(s) \bigr) \, ds  - \int_0^t \vo(x(s)) \, ds\\
&=\int_0^t \Big( \de(x(s))\cdot\nabla_y\Phi(x(s),x(s)/\eps)+\eps \cL_1 \Phi^\eps (x(s))  + \eps^2 \cL_2 \Phi^\eps(x(s)) \Big) \, ds 
   \\ &\quad\quad\quad\quad- \eps \Big(\Phi^\eps(x^\eps (t)) - \Phi^\eps(x^\eps (0)) \Big) + \sqrt{\eps} M^{\eps}(t),
\end{align*}
where
\be\nonumber
M^\eps(t):= \int_0^t \left(\sqrt{2 \eta_0} \nabla_y \Phi^\eps + \eps \sqrt{2 \eta_0} \nabla_x \Phi^\eps \right) dW.
\ene
Since $\Phi(x,y)$ is periodic in both coordinates
we have that
$$
\|\nabla_y\Phi(x,x/\epsilon)\|_{L^{\infty}} \le C,
\quad \|\Phi^\eps (x) \|_{L^{\infty}} \leq C, \quad \| \cL_1 \Phi^\eps \|_{L^{\infty}} \leq C, \quad \| \cL_1 \Phi^\eps \|_{L^{\infty}} \leq C
$$
and
\begin{equation}\label{e:martingale_estim}
\Ex \|M^\eps (t) \|^p \leq C, \quad p \geq 1.
\end{equation}
We combine the above calculations to obtain
\be\nonumber
x^\eps(t) = x(0) + \int_0^t \vo(x^{\eps}(s)) \, ds + H^\eps(t) + \sqrt{\eps} \tilde{M}^{\eps}(t),
\ene
where 
\begin{eqnarray*}
H^{\eps}(t) &:=& -\eps \Big(\Phi^\eps(x^\eps (t)) - \Phi^\eps(x^\eps (0)) \Big) + \int_0^t \left(\delta^{\eps}(x^{\eps}(s)) + \eps \delta_1^{\eps}(x^{\eps}(s)) \right)\, ds
          \\ && + \int_0^t \Big(
\de(x(s))\cdot\nabla_y\Phi(x(s),x(s)/\eps)+
\eps \cL_1 \Phi^\eps (x(s))  + \eps^2 \cL_2 \Phi^\eps(x(s)) \Big) \, ds 
\end{eqnarray*}
and
\be\nonumber
\tilde{M}^{\eps}(t) = M^{\eps}(t) + \sqrt{2\eta_0}W(t).
\ene
Our estimates imply that
\be\nonumber
\lim_{\eps \rightarrow 0} \Ex \sup_{t \in [0,T]} |H^{\eps}(t)| = 0.
\ene
Furthermore, estimate~\eqref{e:martingale_estim}, together with the Burkh\"{o}lder-Davis-Gundy inequality imply that
\be\nonumber
\Ex \sup_{t \in [0,T]} |\tilde{M}^{\eps}(t)| \leq C.
\ene
On the other hand,
\be\nonumber
x(t) = x(0) + \int_0^t \vo(x(s)) \, ds. 
\ene
Set $\theta(T):= \Ex \sup_{t \in [0,T]} | x^{\eps}(t) - x(t) |$. 
Because $v_0$ is periodic it is
in fact globally Lipschitz so that we obtain 
\be\nonumber
\theta(T) \leq C \int_0^T \theta(t) \, dt + h^\eps(T),
\ene
where 
$$
\lim_{\eps \rightarrow 0} h^\eps(T) = 0.
$$
We use Gronwall's inequality to deduce
\be\nonumber
\theta(T) \leq h^\eps \left(1 + C T e^{C T} \right),
\ene
from which the claim follows.
\qed
\end{proof}
%
%

\section*{Appendix 2}
\addcontentsline{toc}{section}{Appendix}
\label{app:b}

In this appendix we study the homogenization problem~\eqref{eq:pde22} in one dimension. In this case we can calculate the homogenized coefficient explicitly and to prove Assumption~\ref{ass:1}. More details can be found in~\cite[Ch. 12]{PS08}.

\subsection*{The Homogenized Equations}

We take $d=1$ in~\eqref{eq:pde22} and set $D = [0,L].$  Then the Dirichlet problem
\eqref{eq:pde22} reduces to a two--point boundary value problem:
\begin{subequations}
\begin{eqnarray}
- \frac{d}{d x} \left( \exp \left( u \left(x,\frac{x}{\eps} \right) \right) \frac{d p^\eps}{dx} \right) & = & f
\quad {\mbox{for}}\,x \in (0,L), \label{e:dir_1_d_e}
\\ p^\eps(0) = p^\eps(L) & = & 0.
\label{e:dir_1d_bc}
\end{eqnarray}
\label{e:dir_1d}
\end{subequations}
We assume that $u(x,y)$ is smooth in both of its arguments and periodic in $y$ with period 1. Furthermore, we assume that this function is bounded from above and below. Consequently, there exist constants $0 < \alpha \leq \beta<\infty$ such that
\begin{equation}
\label{eq:abd}
\alpha \le \exp(u(x,y)) \le \beta, \quad \forall y \in [0,1].
\end{equation}
We also assume that $f$ is smooth.

The cell problem becomes a boundary value problem for an ordinary differential
equation with periodic boundary conditions. Introducing the notation $k(x,y):=\exp(u(x,y))$, the cell problem can be written as
\begin{subequations}
\begin{equation}
- \frac{\partial}{\partial y} \left( k (x,y ) \frac{\partial \chi}{\partial y} \right) =  \frac{\partial k(x,y)}{\partial y},
\quad {\mbox{for}}\; y \, \in (0,1),
\label{e:cell_1d_e}
\end{equation}
\begin{equation}
\chi \, \, \mbox{is } 1 \mbox{--periodic},
\quad \quad \int_0^1 \chi (x,y) \, dy = 0.
\label{e:cell_1d_bc}
\end{equation}
\label{e:cell_1d}
\end{subequations}
Notice that the macrovariable $x$ enters the cell problem~\eqref{e:cell_1d} as a parameter.
Since $d = 1$ we only have one effective coefficient which is given by the
one dimensional version of \eqref{eq:K},\eqref{eq:Q}, namely
\begin{eqnarray}
k_0(x) & = & \int_0^1 \left( k(x,y) + k(x,y) \frac{\partial \chi}{\partial y}(x,y)
\right) \, dy
            \nonumber \\ & = & \left\langle k(x,y) \left(1 + \frac{\partial \chi}{\partial y}(x,y) \right)
        \right\rangle
\label{e:eff_coeff_1d}
\end{eqnarray}
where we have introduced the notation $\langle \phi(x,y) \rangle := \int_0^1 \phi(x,y) \, dy$. The homogenized equation is then
\begin{subequations}
\label{eq:onedh}
\begin{eqnarray}
- \frac{d}{d x} \left( k_0(x) \frac{d p_0}{d x} \right) &=& f, \quad x \in (0,L), \\
p(0) = p(L) &=& 0.
\end{eqnarray}
\end{subequations}

\subsection*{Explicit Solution of the Cell Problem}

Equation \eqref{e:cell_1d_e} can be solved exactly. After integrating the equation and applying the periodic boundary conditions, we obtain
\[
 \chi(x,y) = -y + c_1 \int_0^y \frac{1}{k(x,y)} \, dy + c_2,
\]
with
\[
c_1(x) =  \frac{1}{\int_0^1 \frac{1}{k(x,y)} \, dy} = \langle k(x,y)^{-1}  \rangle^{-1}.
\]
Therefore, from \eqref{e:eff_coeff_1d} we obtain:
\begin{equation}
k_0(x) = \langle k(x,y)^{-1}  \rangle^{-1}.
\label{e:eff_1d}
\end{equation}
The constant $c_2$ is irrelevant. This is the formula which gives the homogenized coefficient in one dimension.
It shows clearly that, even in this simple one--dimensional setting, the
homogenized coefficient is not found by simply averaging the
unhomogenized coefficients over a period of the microstructure.
Rather, the homogenized coefficient is the {\it harmonic average}
of the unhomogenized coefficient. It is quite easy to show that $k_0(x)$ is
bounded from above by the average of $k(x,y)$. Notice that the homogenized coefficient can be written in the form
\be\label{e:coeff_homog}
k_0(x) = e^{u_0(x)}, \quad \mbox{where} \quad u_0(x) = \log \left(\langle \exp (- u(x,y)) \rangle^{-1} \right).
\ene

\subsection*{Error Estimates in $W^{1,\infty}$}

The fact that we can obtain an explicit formula for the solution of the boundary value problem~\eqref{e:dir_1d} as well as for the solution of the cell problem~\eqref{e:cell_1d} enables us to prove Assumption~\ref{ass:1}. 

\begin{prop}\label{prop:conv_1d}
Let $p^\eps(x)$ be the solution of the two-point boundary value problem~\eqref{e:dir_1d} where the log permeability $u(x,y)$ is smooth in both of its arguments and satisfies~\eqref{eq:abd}. Let $k(x,y) = \exp(u(x,y))$ and define
\be\nonumber
\ve (x) = k \left(x,\frac{x}{\eps} \right) \frac{d p^{\eps}}{d x}(x)
\ene
and 
\be\nonumber
V(x,y) = k \left(x,y \right) \left(1 + 
\frac{\partial \chi}{\partial y} \left(x,y\right)\right) \frac{d p_0}{d x}(x),
\ene
where $p_0(x)$ is the solution of the homogenized equation
\eqref{eq:onedh}.
Then
\be\label{e:linfty_estim}
\lim_{\eps \to 0}\|\ve(x) - V(x,x/\epsilon) \|_{L^{\infty}} =0.
\ene
\end{prop}

Notice that, by \eqref{eq:p1}, the corrector 
$p_1(x,y)=\chi(x,y)\frac{d p_0}{d x}(x).$
Hence, using the bound \eqref{eq:abd} from below on $a$, 
together with the definition \eqref{eq:pap} of $\pea$,
this theorem delivers the following immediate corollary:

\begin{corollary}
Under the assumptions of Proposition~\ref{prop:conv_1d} we have
\be\nonumber
\lim_{\eps \to 0}\|p^\eps - \pea\|_{W^{1,\infty}}=0.
\ene
\end{corollary}

\noindent {\it Proof of Proposition~\ref{prop:conv_1d}.}
We have that 
\be\nonumber
\frac{d\chi}{dy}(x,y) = -1 +\frac{k_0(x)}{k(x,y)}.
\ene
Consequently
\be\nonumber
V(x,y) = k_0(x) \frac{d p_0}{d x}(x).
\ene
Define a function $F$ by $F'(z)=f(z).$
We solve the homogenized equation to obtain
\be\nonumber
 k_0(x) \frac{dp_0}{dx}(x) = - F(x) + c, 
\ene
with
$$
c = \frac{\int_0^L k^{-1}_0(z) F(z) \, dz}{\int_0^L k^{-1}_0(z) \, dz}.
$$
Similarly, from~\eqref{e:dir_1d} we obtain
\be\nonumber
k \left(x,\frac{x}{\eps} \right) \frac{d p^\eps}{d x} = - F(x) + c^\eps,
\ene
with
$$
c^\eps = \frac{\int_0^L k^{-1}(z,z/\eps) F(z) \, dz}{\int_0^L k^{-1}(z,z/\eps) \, dz}.
$$

From the above calculations we deduce that
\be\nonumber
\|\ve(x) - V(x,x/\epsilon) \|_{L^{\infty}} = |c - c^\eps |.
\ene
It suffices to show that $|c - c^\eps |={\cal O}(\eps).$ This will follow from the fact that
$$\int_0^L k^{-1}(z,z/\eps) G(z) = \int_0^L k_0^{-1}(z)G(z)dz + {\cal O}(\eps)$$
for any smooth function $G$, as $\epsilon \to 0$. To see this, define integer $N$ and $\delta \in [0,\epsilon)$
uniquely by the identity 
\be\label{e:int_part}
L = N\eps + \delta. 
\ene
Then note that, using the uniform bounds on $k(x,y)$
from below, together with uniform (in $y$) Lipschitz
properties of $a(\cdot,y)$ and $G$, we have for $z_n=n\eps$,
\begin{align*}
\int_0^L k^{-1}(z,z/\eps)G(z)dz&=\sum_{n=0}^{N-1}
\int_{n\eps}^{(n+1)\eps}k^{-1}(z_n,z/\eps)G(z_n)dz+{\cal O}(\eps)\\
&=\sum_{n=0}^{N-1}\int_{n\eps}^{(n+1)\eps}k_0^{-1}(z_n)G(z_n)dz+{\cal O}(\eps)\\
&=
\sum_{n=0}^{N-1}\int_{n\eps}^{(n+1)\eps}k_0^{-1}(z)G(z)dz+{\cal O}(\eps)\\
&=\int_0^L k_0^{-1}(z)G(z)dz +{\cal O}(\eps).
\end{align*}
This completes the proof.
\qed

\def\cprime{$'$} \def\cprime{$'$} \def\cprime{$'$} \def\cprime{$'$}
  \def\cprime{$'$} \def\cprime{$'$} \def\cprime{$'$}
  \def\Rom#1{\uppercase\expandafter{\romannumeral #1}}\def\u#1{{\accent"15
  #1}}\def\Rom#1{\uppercase\expandafter{\romannumeral #1}}\def\u#1{{\accent"15
  #1}}\def\cprime{$'$} \def\cprime{$'$} \def\cprime{$'$} \def\cprime{$'$}
  \def\cprime{$'$} \def\cprime{$'$}

\end{document}